\documentclass[11pt]{article}
\usepackage[T1]{fontenc}
\usepackage[utf8]{inputenc}
\usepackage{lmodern}
\usepackage{amsmath,amssymb,amsthm,mathtools}
\usepackage[margin=1in]{geometry}
\usepackage{microtype}
\usepackage{xcolor}
\definecolor{linkcol}{rgb}{0.10,0.20,0.45}
\usepackage[colorlinks=true,linkcolor=linkcol,citecolor=linkcol,
            urlcolor=linkcol]{hyperref}

\pdftrailerid{}

\theoremstyle{plain}
\newtheorem{theorem}{Theorem}[section]
\newtheorem{lemma}[theorem]{Lemma}
\newtheorem{proposition}[theorem]{Proposition}
\newtheorem{corollary}[theorem]{Corollary}
\theoremstyle{definition}

\theoremstyle{remark}

\allowdisplaybreaks
\title{\vspace{-2em}Mockenhaupt's Three-Term Hardy--Littlewood\\
       Majorant Conjecture}

\newcommand{\afn}[1]{\textsuperscript{\normalfont #1}}
\author{%
\normalsize
\begin{tabular}{c@{\hspace{2.2em}}c@{\hspace{2.2em}}c}
\large Guancheng Pan\afn{1} & \large Chengsong You\afn{1} &
  \large Hengyu Wang\afn{1}
\end{tabular}\\[0.5ex]
\begin{tabular}{c@{\hspace{2.2em}}c@{\hspace{2.2em}}c}
\large Junwei Zhou\afn{2,\dag} & \large Wenjun Zhang\afn{1} &
  \large Yongchao Chen\afn{1,\dag}
\end{tabular}\\[1.4ex]
\small
\textsuperscript{1}Apex Intelligence \quad
\textsuperscript{2}Independent Researcher\\[0.8ex]
\texttt{panguancheng@apexin.ai}, \texttt{youchengsong@apexin.ai},
\texttt{wanghengyu@apexin.ai}\\
\texttt{zhoujunwei@apexin.ai}, \texttt{zhangwenjun@apexin.ai},
\texttt{cyc@apexin.ai}\\[0.6ex]
\textsuperscript{\dag}Corresponding authors%
}
\date{8 September 2026}
\begin{document}
\maketitle
\begin{abstract}
For an integer $k \ge 0$, let $f_k(x) = 1 + e(x) + e((k+2)x)$ and $g_k(x) = 1 + e(x) - e((k+2)x)$ on $\mathbb{T} = \mathbb{R}/\mathbb{Z}$, where $e(x) = e^{2\pi ix}$. Mockenhaupt conjectured that $\|g_k\|_{L^p(\mathbb{T})} > \|f_k\|_{L^p(\mathbb{T})}$ whenever $2k < p < 2k+2$. The conjecture was previously known for $k \le 5$. We give a single analytic proof valid for every $k \ge 4$; in particular, this settles all previously open cases $k \ge 6$ and establishes the conjecture for every $k \ge 0$.

The proof reduces the norm comparison to resonant Fourier coefficients on the two-torus and represents these coefficients, after analytic continuation, by triple-Bessel integrals. Neumann's product formula and the Weber--Schafheitlin formula yield a quantitative positive lower bound for the leading mode, while the remaining odd modes are controlled by a uniform tail estimate. The leading mode is then shown to dominate the tail for every $k \ge 4$.

\textbf{AI Usage.} The mathematical argument of this paper was produced by the auto-research system Apex Math, an AI system built by Apex Intelligence. See Appendix~\ref{app:ai} for the complete AI usage statement.
\end{abstract}

\begin{center}\small
\end{center}

\section{Introduction}\label{sec:intro}

Write $e(x) = e^{2\pi ix}$ and $\mathbb{T} = \mathbb{R}/\mathbb{Z}$, equipped with normalized Lebesgue measure. For $0 < p < \infty$, we use

\[
  \|h\|_{L^p(\mathbb{T})} := \left( \int_{\mathbb{T}} |h(x)|^p dx \right)^{1/p} ,
\]
with the usual quasi-norm interpretation when $p < 1$.

The Hardy--Littlewood majorant problem asks whether

\[
  \left\| \sum_{n \in \Lambda } a_n e(nx) \right\|_{L^p(\mathbb{T})} \le \left\| \sum_{n \in \Lambda } b_n e(nx) \right\|_{L^p(\mathbb{T})}
\]
whenever $\Lambda \subset \mathbb{Z}$ is finite and $|a_n| \le b_n$ for every $n \in \Lambda$. For even integers $p = 2m$, the inequality follows from Parseval's identity applied to the $m$-th powers. For non-even exponents it fails in general; see~\cite{r2,r3,r4,r7}.

Montgomery conjectured that counterexamples also exist when the majorant is idempotent, meaning that its Fourier coefficients belong to $\{0,1\}$, and the comparison polynomial is obtained by changing some of the nonzero coefficients to $-1$. Mockenhaupt and Schlag established the existence of such counterexamples using four-term idempotent polynomials~\cite{r7}. Three is the smallest possible number of terms in such a counterexample: after multiplication by a character, a two-term idempotent has the form $1 + e(Mx)$, with $M \ne 0$, and

\[
  1 - e(Mx) = 1 + e\left( M \left( x + \frac{1}{2M} \right) \right) .
\]
Thus changing the relative sign preserves every $L^p$ norm.

The pair $1 + e(x) \pm e(3x)$ already appears in the work of Hardy and Littlewood at $p = 3$; see also~\cite{r4}. Mockenhaupt proposed that the same three-term pattern works between each pair of consecutive even exponents. More precisely, for an integer $k \ge 0$, set

\[
  N := k+2, \qquad  f_k(x) := 1 + e(x) + e(Nx), \qquad  g_k(x) := 1 + e(x) - e(Nx) .
\]
The conjecture is that

\begin{equation}\label{eq:conj}
  \|g_k\|_{L^p(\mathbb{T})} > \|f_k\|_{L^p(\mathbb{T})} \qquad  (2k < p < 2k+2) .
\end{equation}
The two coefficient sequences have the same absolute values and differ by a single sign.

The choice $N = k+2$ is tied to the exponent interval. For an integer $1 \le m < N$, every exponent in $(1 + z \pm z^N)^m$ has a unique representation $j + N\ell$, with $j, \ell \ge 0$ and $j + \ell \le m$. Indeed, equality of two such representations implies

\[
  N|\ell - \ell '| = |j - j'| \le m < N ,
\]
so $\ell = \ell '$ and $j = j'$. Changing the sign of $z^N$ therefore preserves the absolute values of the coefficients of the $m$-th power. Parseval's identity gives

\[
  \|f_k\|_{L^{2m}(\mathbb{T})} = \|g_k\|_{L^{2m}(\mathbb{T})} , \qquad  1 \le m < N .
\]
At $m = N$, the first collision occurs: the exponent $N$ can arise either from a single factor $z^N$ or from $N$ factors $z$. For $k \ge 1$, the smallest third frequency for which both endpoints $2k$ and $2k+2$ lie in the no-collision range is therefore $N = k+2$. The conjecture is thus extremal in the number of terms and, in this sense, in the placement of the third frequency.

Krenedits proved the cases $k = 0,1,2$ in~\cite{r4}, the cases $k = 3,4$ in~\cite{r5}, and the case $k = 5$ in~\cite{r6}. The case $k = 0$ admits an analytic proof, while the cases $1 \le k \le 5$ rely on rigorous numerical estimates. We give a single analytic argument valid for every $k \ge 4$. In particular, this settles all previously open cases $k \ge 6$ and completes the proof of Mockenhaupt's conjecture.

\begin{theorem}\label{res:1.1}
Let $k \ge 4$ be an integer. For every $2k < p < 2k+2$,

\[
  \|1 + e(x) - e((k+2)x)\|_{L^p(\mathbb{T})} > \|1 + e(x) + e((k+2)x)\|_{L^p(\mathbb{T})} .
\]
\end{theorem}
Combining this with the known small cases gives the full conjecture.

\begin{corollary}\label{res:1.2}
Mockenhaupt's conjecture holds for every integer $k \ge 0$.
\end{corollary}

\begin{proof}
The cases $k = 0,1,2$ follow from~\cite{r4}, and $k = 3$ follows from~\cite{r5}. Theorem~\ref{res:1.1} covers all $k \ge 4$.
\end{proof}

\textbf{Proof strategy.} Write $s = p/2 = k + \alpha$, with $0 < \alpha < 1$, and set

\[
  D(s) := \int_{\mathbb{T}} \left( |g_k(x)|^{2s} - |f_k(x)|^{2s} \right) dx .
\]
We must prove $D(s) > 0$. Let $c_s(m,n)$ denote the Fourier coefficients of

\[
  H(\theta ,\psi )^s := |1 + e(\theta ) + e(\psi )|^{2s} .
\]
Restriction to $(\theta ,\psi ) = (x, Nx)$ selects the modes $m + Nn = 0$. A shift by $1/2$ in the second variable changes the sign of precisely the odd modes, giving

\[
  D(s) = -4 \sum_{\substack{q \ge 1 \\ q \,\mathrm{odd}}} c_s(-Nq, q) .
\]

For $q \ge 1$, define

\[
  I_q(s) := \int_{0}^{\infty } J_{qN}(r) J_{q(N-1)}(r) J_q(r) r^{-2s-1} dr .
\]
The Mellin representation of $J_{0}$, followed by analytic continuation for each fixed mode, yields

\begin{equation}\label{eq:strategy}
  D(s) = \frac{2^{2s+3}\Gamma (s+1)^{2} \sin (\pi \alpha )}{\pi } \sum_{\substack{q \ge 1 \\ q \,\mathrm{odd}}} I_q(s) .
\end{equation}
The prefactor is positive. It remains to show that the leading mode dominates the sum of the remaining odd modes.

The leading mode $I_{1}$ is distinguished by the fact that its two large Bessel orders are adjacent. Neumann's product formula combines them with the nonnegative angular weight $\cos \vartheta$. The resulting integral is evaluated by the Weber--Schafheitlin formula, and Euler's transformation makes its positivity explicit. This gives a quantitative lower bound $I_{1}(s) \ge W_{\mathrm{low}}(k,\alpha ) > 0$. For the other odd modes, elementary Bessel bounds give

\[
  \sum_{\substack{q \ge 3 \\ q \,\mathrm{odd}}} |I_q(s)| \le W_{\mathrm{tail}}(k,\alpha ) .
\]
The explicit bounds satisfy

\[
  \frac{W_{\mathrm{tail}}(k,\alpha )}{W_{\mathrm{low}}(k,\alpha )} < \frac{3}{4} \left( \frac{4}{5} \right)^{k-4} < 1 , \qquad  k \ge 4, \quad  0 \le \alpha \le 1 .
\]
This proves $D(s) > 0$ throughout the required interval.

Throughout the paper, $X \asymp Y$ means that $cY \le X \le CY$ for some constants $c, C > 0$.

\section{From the majorant problem to resonant Fourier modes}\label{sec:reduction}

Fix

\[
  k \ge 4, \qquad  N := k+2 .
\]
For real $s > 0$, set

\[
  A(x) := |f_k(x)|^{2} , \qquad  B(x) := |g_k(x)|^{2} ,
\]
and

\[
  D(s) := \int_{\mathbb{T}} \left( B(x)^s - A(x)^s \right) dx .
\]
In the range relevant to Theorem~\ref{res:1.1}, we write

\[
  s = \frac{p}{2} = k + \alpha , \qquad  0 < \alpha < 1 .
\]
Then

\begin{equation}\label{eq:Dnorm}
  D(s) = \|g_k\|_{L^p(\mathbb{T})}^{p} - \|f_k\|_{L^p(\mathbb{T})}^{p} .
\end{equation}

Define

\[
  P(\theta ,\psi ) := 1 + e(\theta ) + e(\psi ) , \qquad  H(\theta ,\psi ) := |P(\theta ,\psi )|^{2} , \qquad  (\theta ,\psi ) \in \mathbb{T}^{2} .
\]
Then

\begin{equation}\label{eq:curves}
  A(x) = H(x, Nx) , \qquad  B(x) = H\left( x, Nx + \frac{1}{2} \right) .
\end{equation}

To justify restricting the Fourier series of $H^s$ to the curves in~\eqref{eq:curves} and integrating it term by term, we first record the local behavior of $H$ at its zeros. Near a point of $\mathbb{T}^{2}$, distances are understood in local Euclidean coordinates.

\begin{lemma}\label{res:2.1}
The zeros of $H$ are

\[
  z_{+} = \left( \frac{1}{3}, \frac{2}{3} \right) , \qquad  z_{-} = \left( \frac{2}{3}, \frac{1}{3} \right) .
\]
Near either zero,

\begin{equation}\label{eq:quadzero}
  H(\theta ,\psi ) \asymp |(\theta ,\psi ) - z_\pm |^{2} .
\end{equation}
Consequently, for every real $\sigma$,

\begin{equation}\label{eq:L1crit}
  H^\sigma \in L^{1}(\mathbb{T}^{2}) \quad  \Longleftrightarrow \quad  \sigma > -1 .
\end{equation}
\end{lemma}

\begin{proof}
The equation $P(\theta ,\psi ) = 0$ gives the two stated zeros. Regarding $P$ as a map from $\mathbb{R}^{2}$ to $\mathbb{R}^{2}$, its Jacobian determinant at $z_\pm$ is $\pm 2\sqrt{3}\pi^{2}$, so $P$ is a local diffeomorphism there. There exist $c, C, \delta > 0$ such that, for each $z \in \{z_{+}, z_{-}\}$,

\[
  c|(\theta ,\psi ) - z|^{2} \le H(\theta ,\psi ) \le C|(\theta ,\psi ) - z|^{2}
\]
whenever $|(\theta ,\psi ) - z| < \delta$. Thus~\eqref{eq:quadzero} follows, and~\eqref{eq:L1crit} follows from

\[
  \int_{0}^{\delta } r^{2\sigma +1} dr < \infty  \quad  \Longleftrightarrow \quad  \sigma > -1 . \qedhere
\]
\end{proof}

The preceding local description gives the regularity needed for absolute convergence of the Fourier series of $H^s$. For $s \ge 1$ and $(m,n) \in \mathbb{Z}^{2}$, define

\[
  c_s(m,n) := \int_{\mathbb{T}^{2}} H(\theta ,\psi )^s e(-m\theta - n\psi ) d\theta d\psi .
\]

\begin{lemma}\label{res:2.2}
For every real $s \ge 1$,

\[
  \sum_{m,n \in \mathbb{Z}} |c_s(m,n)| < \infty .
\]
Consequently,

\begin{equation}\label{eq:fourierH}
  H(\theta ,\psi )^s = \sum_{m,n \in \mathbb{Z}} c_s(m,n) e(m\theta + n\psi )
\end{equation}
absolutely and uniformly on $\mathbb{T}^{2}$.
\end{lemma}

\begin{proof}
Set

\[
  u(\theta ,\psi ) := H(\theta ,\psi )^s .
\]
Away from the two zeros of $H$, the function $u$ is smooth. Near either zero, $P$ provides smooth local coordinates in which

\[
  u = |w|^{2s} .
\]
Since

\[
  D^{2}(|w|^{2s}) = O(|w|^{2s-2}) ,
\]
the second derivatives of $u$ are bounded near the zeros for $s \ge 1$. Therefore

\[
  u \in H^{2}(\mathbb{T}^{2}) .
\]

By the Fourier characterization of $H^{2}(\mathbb{T}^{2})$,

\[
  \sum_{m,n \in \mathbb{Z}} (1 + m^{2} + n^{2})^{2} |c_s(m,n)|^{2} < \infty .
\]
Hence, by the Cauchy--Schwarz inequality,

\begin{align*}
  \sum_{m,n \in \mathbb{Z}} |c_s(m,n)| &\le \left( \sum_{m,n \in \mathbb{Z}} (1 + m^{2} + n^{2})^{-2} \right)^{1/2} \\
  &\qquad \times \left( \sum_{m,n \in \mathbb{Z}} (1 + m^{2} + n^{2})^{2} |c_s(m,n)|^{2} \right)^{1/2} < \infty .
\end{align*}
Hence the Fourier series converges absolutely and uniformly to a continuous function $v$. On the other hand, its partial sums converge to $u$ in $L^{2}(\mathbb{T}^{2})$. Thus $u = v$ almost everywhere. Since both $u$ and $v$ are continuous, they agree everywhere. This proves~\eqref{eq:fourierH}.
\end{proof}

The absolute and uniform convergence above allows us to restrict the Fourier series of $H(\theta ,\psi )^s$ to the two curves in~\eqref{eq:curves} and integrate term by term.

\begin{proposition}[Resonant Fourier modes]\label{res:2.3}
For every real $s \ge 1$,

\begin{equation}\label{eq:resonant}
  D(s) = -4 \sum_{\substack{q \ge 1 \\ q \,\mathrm{odd}}} c_s(-Nq, q) .
\end{equation}
\end{proposition}

\begin{proof}
By Lemma~\ref{res:2.2},

\[
  H(\theta ,\psi )^s = \sum_{m,n \in \mathbb{Z}} c_s(m,n) e(m\theta + n\psi )
\]
absolutely and uniformly. Using~\eqref{eq:curves},

\[
  \int_{\mathbb{T}} A(x)^s dx = \sum_{m,n \in \mathbb{Z}} c_s(m,n) \int_{\mathbb{T}} e((m + Nn)x) dx = \sum_{q \in \mathbb{Z}} c_s(-Nq, q) .
\]
Similarly,

\begin{align*}
  \int_{\mathbb{T}} B(x)^s dx &= \sum_{m,n \in \mathbb{Z}} c_s(m,n) e(n/2) \int_{\mathbb{T}} e((m + Nn)x) dx \\
  &= \sum_{q \in \mathbb{Z}} (-1)^q c_s(-Nq, q) .
\end{align*}
Therefore

\[
  D(s) = -2 \sum_{\substack{q \in \mathbb{Z} \\ q \,\mathrm{odd}}} c_s(-Nq, q) .
\]
Since $H(-\theta ,-\psi ) = H(\theta ,\psi )$, we have $c_s(-m,-n) = c_s(m,n)$, and~\eqref{eq:resonant} follows.
\end{proof}

\section{Bessel representation of the resonant modes}\label{sec:bessel}

By Proposition~\ref{res:2.3}, the norm difference is determined by the resonant Fourier coefficients $c_s(-Nq, q)$. In this section we represent these coefficients by triple-Bessel integrals. The identity is first established in a strip where the Mellin representation and Fubini's theorem apply, and is then continued holomorphically, for each fixed $q$, to the positive range required below.

Throughout this section, $J_\nu$ denotes the Bessel function of the first kind of order $\nu$.

\subsection{Holomorphic dependence on the exponent}\label{sub:3.1}
For $\Re s > -1$, define

\[
  H(\theta ,\psi )^s := \exp \big( s \log H(\theta ,\psi ) \big)
\]
when $H(\theta ,\psi ) > 0$, and set $H^s = 0$ at the zeros of $H$. For $(m,n) \in \mathbb{Z}^{2}$, let

\[
  c_s(m,n) := \int_{\mathbb{T}^{2}} H(\theta ,\psi )^s e(-m\theta - n\psi ) d\theta d\psi .
\]
By Lemma~\ref{res:2.1}, this integral is absolutely convergent. The convention at the zeros does not affect its value. For real $s \ge 1$, it agrees with the definition in Section~\ref{sec:reduction}.

\begin{lemma}\label{res:3.1}
For every $(m,n) \in \mathbb{Z}^{2}$, the function

\[
  s \longmapsto c_s(m,n)
\]
is holomorphic in the half-plane

\[
  \Re s > -1 .
\]
\end{lemma}

\begin{proof}
Let $K \Subset \{\Re s > -1\}$, and put

\[
  \sigma_{0} := \min_{s \in K} \Re s , \qquad  \sigma_{1} := \max_{s \in K} \Re s .
\]
For $s \in K$ and $H > 0$,

\[
  |H^s| = H^{\Re s} \le H^{\sigma_{0}} + H^{\sigma_{1}} .
\]
The right-hand side is integrable by Lemma~\ref{res:2.1}, since $\sigma_{0} > -1$. Dominated convergence gives continuity of $c_s(m,n)$.

Let $T$ be a triangle whose closure lies in this half-plane. The same bound, with $K = \partial T$, justifies Fubini's theorem:

\begin{align*}
  \int_{\partial T} c_s(m,n) ds &= \int_{\mathbb{T}^{2}} e(-m\theta - n\psi ) \left( \int_{\partial T} H(\theta ,\psi )^s ds \right) d\theta d\psi \\
  &= 0 .
\end{align*}
The inner integral vanishes because $s \mapsto H^s$ is entire where $H > 0$, and is the constant zero function at the zeros of $H$ under our convention. Morera's theorem proves the lemma.
\end{proof}

\subsection{The triple-Bessel coefficient}\label{sub:3.2}
We next compute the Fourier coefficients of $J_{0}(r\sqrt{H})$ directly from the integral representation of $J_{0}$.

\begin{lemma}\label{res:3.2}
For every $r > 0$ and $(m,n) \in \mathbb{Z}^{2}$,

\begin{equation}\label{eq:triple}
  \int_{\mathbb{T}^{2}} J_{0}\big( r\sqrt{H(\theta ,\psi )} \big) e(-m\theta - n\psi ) d\theta d\psi = (-1)^{m+n} J_m(r) J_n(r) J_{m+n}(r) .
\end{equation}
In particular, for $q \ge 1$,

\begin{equation}\label{eq:tripleq}
  \int_{\mathbb{T}^{2}} J_{0}(r\sqrt{H}) e(Nq\theta - q\psi ) d\theta d\psi = (-1)^{Nq} J_{qN}(r) J_{q(N-1)}(r) J_q(r) .
\end{equation}
\end{lemma}

\begin{proof}
Write $P(\theta ,\psi ) = 1 + e(\theta ) + e(\psi )$, so that $\sqrt{H} = |P|$. The integral representation of $J_{0}$ gives

\[
  J_{0}(r|P|) = \frac{1}{2\pi } \int_{0}^{2\pi } \exp \big( ir \Re (e^{-it}P) \big) dt .
\]
Since

\[
  \Re (e^{-it}P) = \cos t + \cos (2\pi \theta - t) + \cos (2\pi \psi - t) ,
\]
Fubini's theorem gives

\begin{align*}
  \int_{\mathbb{T}^{2}} J_{0}(r\sqrt{H}) e(-m\theta - n\psi ) d\theta d\psi
  &= \frac{1}{2\pi } \int_{0}^{2\pi } e^{ir\cos t} A_m(t) A_n(t) dt ,
\end{align*}
where

\[
  A_m(t) := \int_{\mathbb{T}} e^{ir\cos (2\pi \theta - t)} e(-m\theta ) d\theta .
\]
The interchange is justified because the integrand has modulus $1$ on the finite-measure product space.

For each integer $\ell$, Bessel's integral formula~\cite[Eq.~10.9.2]{r1} gives

\[
  \frac{1}{2\pi } \int_{0}^{2\pi } e^{ir\cos u} e^{-i\ell u} du = i^\ell J_\ell (r) .
\]
The change of variables $u = 2\pi \theta - t$, followed by periodicity, therefore gives $A_m(t) = i^m J_m(r) e^{-imt}$. Consequently,

\begin{align*}
  \int_{\mathbb{T}^{2}} J_{0}(r\sqrt{H}) e(-m\theta - n\psi ) d\theta d\psi
  &= i^{m+n} J_m(r) J_n(r) \frac{1}{2\pi } \int_{0}^{2\pi } e^{ir\cos t} e^{-i(m+n)t} dt \\
  &= (-1)^{m+n} J_m(r) J_n(r) J_{m+n}(r) .
\end{align*}
Taking $(m,n) = (-Nq, q)$ and using $J_{-\ell } = (-1)^\ell J_\ell$ gives~\eqref{eq:tripleq}.
\end{proof}

\subsection{Mode-by-mode analytic continuation}\label{sub:3.3}
For $q \ge 1$, set

\[
  \Psi_q(r) := J_{qN}(r) J_{q(N-1)}(r) J_q(r) ,
\]
and define

\begin{equation}\label{eq:Iqdef}
  I_q(s) := \int_{0}^{\infty } \Psi_q(r) r^{-2s-1} dr
\end{equation}
whenever the integral is absolutely convergent. We first establish a strip of absolute convergence in which the mode-by-mode continuation will be carried out.

\begin{lemma}\label{res:3.3}
For $q \ge 1$, set

\[
  C_q := \Gamma (qN+1) \Gamma (q(N-1)+1) \Gamma (q+1) .
\]
Then

\begin{equation}\label{eq:Psibound}
  |\Psi_q(r)| \le \frac{(r/2)^{2qN}}{C_q} , \qquad  r > 0 ,
\end{equation}
and, for each fixed $q$,

\[
  \Psi_q(r) = O_q(r^{-3/2}) \qquad  (r \to \infty ) .
\]
Consequently, $I_q$ is holomorphic on

\[
  \Omega_q := \left\{ s \in \mathbb{C} : -\frac{3}{4} < \Re s < qN \right\} .
\]
\end{lemma}

\begin{proof}
The three orders satisfy

\[
  qN + q(N-1) + q = 2qN .
\]
Poisson's integral representation~\cite[Eq.~10.9.4]{r1} gives, for $\ell \ge 0$,

\[
  J_\ell (r) = \frac{2(r/2)^\ell }{\sqrt{\pi }\, \Gamma (\ell + \tfrac{1}{2})} \int_{0}^{1} (1-t^{2})^{\ell - \frac{1}{2}} \cos (rt) dt .
\]
Hence

\[
  |J_\ell (r)| \le \frac{2(r/2)^\ell }{\sqrt{\pi }\, \Gamma (\ell + \tfrac{1}{2})} \int_{0}^{1} (1-t^{2})^{\ell - \frac{1}{2}} dt .
\]
Since

\[
  \int_{0}^{1} (1-t^{2})^{\ell - \frac{1}{2}} dt = \frac{\sqrt{\pi }\, \Gamma (\ell + \tfrac{1}{2})}{2\Gamma (\ell +1)} ,
\]
we obtain

\[
  |J_\ell (r)| \le \frac{(r/2)^\ell }{\Gamma (\ell +1)} , \qquad  r > 0 .
\]

Multiplying the three estimates gives~\eqref{eq:Psibound}. For each fixed $\ell$, the large-argument estimate~\cite[Eq.~10.17.3]{r1}

\[
  J_\ell (r) = O_\ell (r^{-1/2}) \qquad  (r \to \infty )
\]
gives $\Psi_q(r) = O_q(r^{-3/2})$.

Let $K \Subset \Omega_q$, and set

\[
  \sigma_{0} := \min_{s \in K} \Re s , \qquad  \sigma_{1} := \max_{s \in K} \Re s .
\]
Then $\sigma_{0} > -3/4$ and $\sigma_{1} < qN$. For $s \in K$, \eqref{eq:Psibound} gives

\[
  \left| \Psi_q(r) r^{-2s-1} \right| \le \frac{2^{-2qN}}{C_q} r^{2qN - 2\sigma_{1} - 1} , \qquad  0 < r \le 1 ,
\]
whereas the large-argument bound gives

\[
  \left| \Psi_q(r) r^{-2s-1} \right| \le C_{q,K} r^{-2\sigma_{0} - 5/2} , \qquad  r \ge 1 .
\]
Both majorants are integrable. Hence $I_q(s)$ is absolutely convergent on $\Omega_q$, locally uniformly in $s$. Since $s \mapsto r^{-2s-1}$ is entire for each $r > 0$, the standard holomorphic parameter-integral theorem shows that $I_q$ is holomorphic on $\Omega_q$.
\end{proof}

The Mellin integral representing $H(\theta ,\psi )^s$ is absolutely convergent only in the strip

\[
  -\frac{1}{4} < \Re s < 0 .
\]
After taking the $q$-th resonant Fourier coefficient, however, the triple-Bessel product vanishes to order $2qN$ at the origin. Consequently, $I_q(s)$ is holomorphic on the larger strip $\Omega_q$, which contains the positive exponents required below. This allows the identity to be continued mode by mode.

Set

\[
  C(s) := \frac{2^{2s+1}\Gamma (s+1)}{\Gamma (-s)} .
\]
The precise relation between $c_s(-Nq, q)$ and $I_q(s)$ is the following.

\begin{proposition}\label{res:3.4}
For every $q \ge 1$ and every $s \in \Omega_q$,

\begin{equation}\label{eq:cI}
  c_s(-Nq, q) = C(s) (-1)^{Nq} I_q(s) .
\end{equation}
\end{proposition}

\begin{proof}
We begin in the strip

\[
  -\frac{1}{4} < \Re s < 0 .
\]
There the Mellin formula for $J_{0}$~\cite[Eq.~10.22.43]{r1} gives

\begin{equation}\label{eq:mellin}
  a^{2s} = C(s) \int_{0}^{\infty } J_{0}(ar) r^{-2s-1} dr , \qquad  a > 0 .
\end{equation}
Taking $a = \sqrt{H(\theta ,\psi )}$, we obtain

\[
  H(\theta ,\psi )^s = C(s) \int_{0}^{\infty } J_{0}\big( r\sqrt{H(\theta ,\psi )} \big) r^{-2s-1} dr
\]
whenever $H(\theta ,\psi ) > 0$, and hence almost everywhere on $\mathbb{T}^{2}$.

It remains to justify the change in the order of integration. Put $\sigma = \Re s$ and

\[
  M(\sigma ) := \int_{0}^{\infty } |J_{0}(t)| t^{-2\sigma -1} dt .
\]
The estimates $J_{0}(t) = O(1)$ as $t \to 0$ and $J_{0}(t) = O(t^{-1/2})$ as $t \to \infty$ show that $M(\sigma ) < \infty$ when $-1/4 < \sigma < 0$. Let

\[
  Z := \{ (\theta ,\psi ) \in \mathbb{T}^{2} : H(\theta ,\psi ) = 0 \} .
\]
For $(\theta ,\psi ) \notin Z$, the change of variables $t = r\sqrt{H}$ gives

\[
  \int_{0}^{\infty } |J_{0}(r\sqrt{H})| r^{-2\sigma -1} dr = H^\sigma M(\sigma ) .
\]
Consequently, Lemma~\ref{res:2.1} gives the absolute bound

\begin{align*}
  \int_{\mathbb{T}^{2}\setminus Z} \int_{0}^{\infty } |J_{0}(r\sqrt{H})| r^{-2\sigma -1} dr d\theta d\psi
  = M(\sigma ) \int_{\mathbb{T}^{2}} H^\sigma d\theta d\psi < \infty .
\end{align*}
Thus Fubini's theorem applies. Since $Z$ has measure zero, it may be restored in the inner integral after the interchange, and we obtain

\begin{align*}
  c_s(-Nq, q) &= C(s) \int_{0}^{\infty } r^{-2s-1} \left( \int_{\mathbb{T}^{2}} J_{0}(r\sqrt{H}) e(Nq\theta - q\psi ) d\theta d\psi \right) dr \\
  &= C(s) (-1)^{Nq} I_q(s) ,
\end{align*}
where the last equality follows from Lemma~\ref{res:3.2}. Hence~\eqref{eq:cI} holds in the anchor strip $-1/4 < \Re s < 0$.

By Lemma~\ref{res:3.1}, $c_s(-Nq, q)$ is holomorphic on $\Omega_q$, while Lemma~\ref{res:3.3} gives the same for $I_q(s)$. Moreover,

\[
  C(s) = \frac{2^{2s+1}\Gamma (s+1)}{\Gamma (-s)}
\]
is holomorphic there: $1/\Gamma (-s)$ is entire, and $\Gamma (s+1)$ has no pole in $\Re s > -3/4$. Since $\Omega_q$ is connected and contains the anchor strip, the identity theorem extends~\eqref{eq:cI} to all of $\Omega_q$.
\end{proof}

\begin{theorem}[Bessel representation]\label{res:3.5}
Let

\[
  k \ge 4, \qquad  s = k + \alpha , \qquad  0 < \alpha < 1 .
\]
Then

\begin{equation}\label{eq:Drep}
  D(s) = \frac{2^{2s+3}\Gamma (s+1)^{2}\sin (\pi \alpha )}{\pi } \sum_{\substack{q \ge 1 \\ q \,\mathrm{odd}}} I_q(s) ,
\end{equation}
and the series converges absolutely.
\end{theorem}

\begin{proof}
Since

\[
  s < k+1 = N-1 < N \le qN ,
\]
we have $s \in \Omega_q$ for every $q \ge 1$, so Proposition~\ref{res:3.4} applies to every resonant mode. Moreover, if $q$ is odd, then

\[
  (-1)^{Nq} = (-1)^{N} = (-1)^{k} .
\]
Proposition~\ref{res:2.3} therefore gives

\[
  D(s) = -4C(s)(-1)^{k} \sum_{\substack{q \ge 1 \\ q \,\mathrm{odd}}} I_q(s) .
\]

By the reflection formula,

\[
  \frac{1}{\Gamma (-s)} = -\frac{\Gamma (s+1)\sin (\pi s)}{\pi } ,
\]
and hence

\[
  C(s) = -\frac{2^{2s+1}\Gamma (s+1)^{2}\sin (\pi s)}{\pi } .
\]
Since $\sin (\pi s) = \sin \big( \pi (k+\alpha ) \big) = (-1)^{k}\sin (\pi \alpha )$, we obtain~\eqref{eq:Drep}.

Finally, $0 < \alpha < 1$ implies $C(s) \ne 0$. Using~\eqref{eq:cI} and the absolute summability of the Fourier coefficients from Lemma~\ref{res:2.2}, we have

\[
  \sum_{\substack{q \ge 1 \\ q \,\mathrm{odd}}} |I_q(s)| = \frac{1}{|C(s)|} \sum_{\substack{q \ge 1 \\ q \,\mathrm{odd}}} |c_s(-Nq, q)| \le \frac{1}{|C(s)|} \sum_{m,n \in \mathbb{Z}} |c_s(m,n)| < \infty . \qedhere
\]
\end{proof}

The prefactor in~\eqref{eq:Drep} is positive. Thus the proof of Theorem~\ref{res:1.1} is reduced to showing that

\begin{equation}\label{eq:target}
  \sum_{\substack{q \ge 1 \\ q \,\mathrm{odd}}} I_q(s) > 0 .
\end{equation}

\section{The leading mode}\label{sec:leading}

Throughout this section,

\[
  k \ge 4, \qquad  N = k+2, \qquad  0 \le \alpha \le 1, \qquad  s = k + \alpha .
\]
We consider the leading term

\[
  I_{1}(s) = \int_{0}^{\infty } J_N(r) J_{N-1}(r) J_{1}(r)\, r^{-2s-1} dr .
\]
Set

\[
  L := 2N - 1 = 2k+3 .
\]

The adjacent orders $N$ and $N-1$ allow us to use Neumann's product formula with the nonnegative weight $\cos \vartheta$. This reduces the positivity question to an integral of two Bessel functions, which we evaluate below.

\begin{proposition}\label{res:4.1}
For $0 \le \alpha \le 1$,

\begin{equation}\label{eq:neumann}
  I_{1}(s) = \frac{2}{\pi } \int_{0}^{\pi /2} \cos \vartheta \, W(2\cos \vartheta ) d\vartheta ,
\end{equation}
where

\[
  W(A) := \int_{0}^{\infty } J_L(Ar) J_{1}(r)\, r^{-2s-1} dr , \qquad  0 \le A \le 2 .
\]
\end{proposition}

\begin{proof}
Neumann's product formula~\cite[Eq.~10.9.26]{r1} (see also~\cite[Eq.~(5.43.1), p.~150]{r8}) gives

\[
  J_N(r) J_{N-1}(r) = \frac{2}{\pi } \int_{0}^{\pi /2} J_L(2r\cos \vartheta ) \cos \vartheta \, d\vartheta .
\]
Its order condition is automatic here, since $N + (N-1) > -1$.

We verify absolute integrability before changing the order of integration. For $0 < r \le 1$, uniformly in $0 \le \vartheta \le \pi /2$,

\[
  |J_L(2r\cos \vartheta )| \le \frac{r^{L}}{\Gamma (L+1)} , \qquad  |J_{1}(r)| \le \frac{r}{2} .
\]
Hence

\[
  \left| J_L(2r\cos \vartheta ) J_{1}(r) \cos \vartheta \, r^{-2s-1} \right| \le \frac{1}{2\Gamma (L+1)} r^{L-2s} = \frac{1}{2\Gamma (L+1)} r^{3-2\alpha } ,
\]
which is integrable at $0$ for $0 \le \alpha \le 1$. For $r \ge 1$, the bound~\cite[Eq.~10.14.1]{r1}

\[
  |J_n(x)| \le 1 , \qquad  x \in \mathbb{R}, \quad  n = 0,1,2,\dots ,
\]
gives

\[
  \left| J_L(2r\cos \vartheta ) J_{1}(r) \cos \vartheta \, r^{-2s-1} \right| \le r^{-2s-1} ,
\]
which is integrable at infinity. Thus the corresponding integral over $(0,\infty ) \times (0,\pi /2)$ is absolutely convergent, and Fubini's theorem yields~\eqref{eq:neumann}.
\end{proof}

To evaluate $W(A)$, we use the following form of the Weber--Schafheitlin formula~\cite[Eq.~10.22.56]{r1} (see also~\cite[Section~13.4]{r8}). If

\[
  0 < a < b , \qquad  \Re (\mu + \nu + 1) > \Re \lambda > -1 ,
\]
then

\begin{align}
  \int_{0}^{\infty } J_\mu (ar) J_\nu (br) r^{-\lambda } dr
  &= \frac{a^\mu \Gamma \left( \frac{\mu +\nu -\lambda +1}{2} \right)}{2^\lambda b^{\mu -\lambda +1} \Gamma \left( \frac{\nu -\mu +\lambda +1}{2} \right) \Gamma (\mu +1)} \notag \\
  &\qquad \times {}_{2}F_{1}\left( \frac{\mu +\nu -\lambda +1}{2}, \frac{\mu -\nu -\lambda +1}{2}; \mu +1; \frac{a^{2}}{b^{2}} \right) . \label{eq:ws}
\end{align}
Here

\[
  {}_{2}F_{1}(a,b;c;z) = \sum_{j=0}^{\infty } \frac{(a)_j (b)_j}{(c)_j\, j!} z^j , \qquad  |z| < 1 ,
\]
where

\[
  (a)_{0} = 1 , \qquad  (a)_j = a(a+1)\cdots (a+j-1) , \qquad  j \ge 1 .
\]
Formula~\eqref{eq:ws} is~\cite[Eq.~10.22.56]{r1} written in terms of the usual Gauss hypergeometric function. Indeed, the DLMF uses Olver's normalized function

\[
  \mathbf{F}(a,b;c;z) = \frac{{}_{2}F_{1}(a,b;c;z)}{\Gamma (c)}
\]
(cf.~\cite[Eq.~15.1.2]{r1}), which accounts for the factor $\Gamma (\mu +1)^{-1}$ in~\eqref{eq:ws}.

\begin{proposition}\label{res:4.2}
Let $0 \le \alpha \le 1$ and $s = k + \alpha$. If $0 < A < 1$, then

\begin{equation}\label{eq:Wsmall}
  W(A) = \frac{A^{L}\Gamma (2-\alpha )}{2^{2s+1}\Gamma (\alpha )\Gamma (L+1)}\, {}_{2}F_{1}\left( 2-\alpha , 1-\alpha ; L+1; A^{2} \right) ,
\end{equation}
where the value at $\alpha = 0$ is understood by continuity.

If $1 < A \le 2$, then

\begin{equation}\label{eq:Wbig}
  W(A) = \frac{\Gamma (2-\alpha )A^{2s-1}}{2^{2s+1}\Gamma (2k+2+\alpha )} (1 - A^{-2})^{2s+1}\, {}_{2}F_{1}\left( \alpha , 2k+3+\alpha ; 2; A^{-2} \right) .
\end{equation}
\end{proposition}

\begin{proof}
In both applications of~\eqref{eq:ws}, the orders are $L = 2k+3$ and $1$, and $\lambda = 2s+1$. Thus

\[
  \mu + \nu + 1 - \lambda = 4 - 2\alpha > 0 , \qquad  \lambda > -1 ,
\]
so the order conditions hold for $0 \le \alpha \le 1$.

For $0 < A < 1$, take $(a,b,\mu ,\nu ,\lambda ) = (A, 1, L, 1, 2s+1)$. The two upper hypergeometric parameters are $2-\alpha$ and $1-\alpha$; the lower parameter is $L+1$, and the other Gamma factor in the denominator is $\Gamma (\alpha )$. This gives~\eqref{eq:Wsmall} for $0 < \alpha \le 1$. At $\alpha = 0$, the defining integral for $W(A)$ is continuous in $\alpha$: its integrand is bounded uniformly by $Cr$ on $0 < r \le 1$ and by $r^{-2k-1}$ on $r \ge 1$. For fixed $A < 1$, the hypergeometric factor remains finite as $\alpha \to 0^{+}$, while $1/\Gamma (\alpha ) \to 0$. Hence $W(A) = 0$ at $\alpha = 0$.

For $1 < A \le 2$, interchange the two Bessel factors and take $(a,b,\mu ,\nu ,\lambda ) = (1, A, 1, L, 2s+1)$. Formula~\eqref{eq:ws} gives

\[
  W(A) = \frac{\Gamma (2-\alpha )A^{2s-1}}{2^{2s+1}\Gamma (2k+2+\alpha )}\, {}_{2}F_{1}\left( 2-\alpha , -2k-1-\alpha ; 2; A^{-2} \right) .
\]
Apply Euler's transformation~\cite[Eq.~15.8.1]{r1}

\[
  {}_{2}F_{1}(a,b;c;x) = (1-x)^{c-a-b}\, {}_{2}F_{1}(c-a, c-b; c; x) , \qquad  0 < x < 1 ,
\]
with $a = 2-\alpha$, $b = -2k-1-\alpha$ and $c = 2$. Since

\[
  c - a = \alpha , \qquad  c - b = 2k+3+\alpha , \qquad  c - a - b = 2s+1 ,
\]
we obtain~\eqref{eq:Wbig}.
\end{proof}

The value $A = 1$ does not affect~\eqref{eq:neumann}, since $A = 1$ corresponds to the single point $\vartheta = \pi /3$. We use~\eqref{eq:Wsmall} to show that the contribution from $0 < A < 1$ is nonnegative, and~\eqref{eq:Wbig} to obtain a quantitative lower bound from the region $1 < A < 2$.

\begin{proposition}\label{res:4.3}
For every $0 \le \alpha \le 1$,

\begin{equation}\label{eq:Wlowbound}
  I_{1}(s) \ge W_{\mathrm{low}}(k,\alpha ) > 0 ,
\end{equation}
where

\begin{equation}\label{eq:Wlowdef}
  W_{\mathrm{low}}(k,\alpha ) := \frac{\Gamma (2-\alpha )}{\pi \Gamma (2k+2+\alpha )(2s+2)} \cdot \frac{2}{5\sqrt{3}} \cdot \frac{9}{8} \left( \frac{9}{16} \right)^{s} .
\end{equation}
\end{proposition}

\begin{proof}
For $0 < \alpha \le 1$, all factors in~\eqref{eq:Wsmall} are nonnegative when $0 < A < 1$: the two upper parameters of the hypergeometric series are $2 - \alpha > 0$ and $1 - \alpha \ge 0$, while its lower parameter is $L+1 > 0$. At $\alpha = 0$, the limiting value is $W(A) = 0$. Thus

\[
  W(A) \ge 0 , \qquad  0 < A < 1 ,
\]
and the corresponding part of~\eqref{eq:neumann} may be discarded.

For $1 < A < 2$, every coefficient of the hypergeometric series in~\eqref{eq:Wbig} is nonnegative. Hence

\[
  {}_{2}F_{1}\left( \alpha , 2k+3+\alpha ; 2; A^{-2} \right) \ge 1 .
\]
Using $A = 2\cos \vartheta$ in~\eqref{eq:neumann} and retaining only the region $1 < A < 2$, we obtain

\[
  I_{1}(s) \ge \frac{\Gamma (2-\alpha )}{\pi \, 2^{2s+1}\Gamma (2k+2+\alpha )} \int_{1}^{2} \frac{A^{2s}(1-A^{-2})^{2s+1}}{\sqrt{4-A^{2}}} dA .
\]
Now set

\[
  t := A - \frac{1}{A} .
\]
Since

\[
  A^{2s}(1-A^{-2})^{2s+1} = \frac{t^{2s+1}}{A} , \qquad  \frac{dA}{A} = \frac{A}{A^{2}+1} dt ,
\]
we have

\[
  \int_{1}^{2} \frac{A^{2s}(1-A^{-2})^{2s+1}}{\sqrt{4-A^{2}}} dA = \int_{0}^{3/2} t^{2s+1} \frac{A}{A^{2}+1} \frac{dt}{\sqrt{4-A^{2}}} .
\]
For $1 < A < 2$,

\[
  \frac{A}{A^{2}+1} \ge \frac{2}{5} , \qquad  \frac{1}{\sqrt{4-A^{2}}} \ge \frac{1}{\sqrt{3}} ,
\]
and therefore

\[
  \int_{1}^{2} \frac{A^{2s}(1-A^{-2})^{2s+1}}{\sqrt{4-A^{2}}} dA \ge \frac{2}{5\sqrt{3}} \cdot \frac{(3/2)^{2s+2}}{2s+2} .
\]
Finally,

\[
  \frac{(3/2)^{2s+2}}{2^{2s+1}} = \frac{9}{8} \left( \frac{9}{16} \right)^{s} ,
\]
which gives~\eqref{eq:Wlowbound}.
\end{proof}

We have obtained a positive quantitative lower bound for the leading mode. The next section estimates the sum of the remaining odd modes.

\section{The tail}\label{sec:tail}

Section~\ref{sec:leading} gives a positive lower bound for the leading mode. We now estimate the remaining odd modes:

\[
  \sum_{\substack{q \ge 3 \\ q \,\mathrm{odd}}} |I_q(s)| .
\]
Throughout this section,

\[
  k \ge 4, \qquad  N = k+2 \ge 6, \qquad  s = k+\alpha , \qquad  0 \le \alpha \le 1 .
\]
In particular,

\[
  0 < s \le N-1 < qN , \qquad  q \ge 1 .
\]
Recall that

\[
  C_q = \Gamma (qN+1)\Gamma (q(N-1)+1)\Gamma (q+1) ,
\]
and set

\[
  R_q := 2C_q^{1/(2qN)} .
\]
By~\eqref{eq:Psibound},

\[
  |\Psi_q(r)| \le \left( \frac{r}{R_q} \right)^{2qN} .
\]
On the other hand, the standard bound

\[
  |J_\ell (r)| \le 1 , \qquad  \ell = 0,1,2,\dots ,
\]
gives $|\Psi_q(r)| \le 1$. Hence

\[
  |\Psi_q(r)| \le \min \left\{ 1, \left( \frac{r}{R_q} \right)^{2qN} \right\} .
\]

\begin{lemma}\label{res:5.1}
For every integer $q \ge 1$,

\[
  |I_q(s)| \le R_q^{-2s} \left( \frac{1}{2qN - 2s} + \frac{1}{2s} \right) .
\]
\end{lemma}

\begin{proof}
Since $s > 0$ and $qN > s$, splitting the integral at $R_q$ gives

\begin{align*}
  |I_q(s)| &\le \int_{0}^{R_q} \left( \frac{r}{R_q} \right)^{2qN} r^{-2s-1} dr + \int_{R_q}^{\infty } r^{-2s-1} dr \\
  &= R_q^{-2s} \left( \frac{1}{2qN - 2s} + \frac{1}{2s} \right) . \qedhere
\end{align*}
\end{proof}

To make the $q$-dependence explicit, we bound $R_q$ linearly from below.

\begin{lemma}\label{res:5.2}
For every $N \ge 6$ and every integer $q \ge 1$,

\[
  R_q \ge \frac{29}{50} Nq .
\]
\end{lemma}

\begin{proof}
Using $n! \ge \left( \frac{n}{e} \right)^{n}$, $(n \ge 1)$, we obtain

\[
  C_q^{1/(2qN)} \ge \frac{q}{e} N^{1/2}(N-1)^{(N-1)/(2N)} .
\]
Hence

\[
  \frac{R_q}{Nq} \ge \frac{2}{e} \left( \frac{(1-1/N)^{N-1}}{N} \right)^{1/(2N)} \ge \frac{2}{e}(eN)^{-1/(2N)} ,
\]
where we use

\[
  \left( 1 - \frac{1}{N} \right)^{N-1} \ge e^{-1} .
\]
Set

\[
  F(N) := \frac{2}{e}(eN)^{-1/(2N)} .
\]
Since

\[
  \frac{d}{dN} \log F(N) = \frac{\log N}{2N^{2}} > 0 ,
\]
$F$ is increasing for $N > 1$. Thus it is enough to check $N = 6$.

Since

\[
  F(6) > \frac{29}{50} \quad  \Longleftrightarrow \quad  6\left( \frac{29}{50} \right)^{12} e^{13} < 2^{12} ,
\]
it remains to verify the latter inequality. First,

\[
  e = \sum_{j=0}^{\infty } \frac{1}{j!} < \frac{65}{24} + \frac{1}{120}\sum_{j=0}^{\infty } 6^{-j} = \frac{65}{24} + \frac{1}{100} < \frac{68}{25} .
\]
Moreover,

\[
  \left( \frac{29}{50} \right)^{12} < \frac{3}{2000} , \qquad  \left( \frac{68}{25} \right)^{13} < 450000 ,
\]
which follow, for instance, from

\[
  \left( \frac{29}{50} \right)^{4} < \frac{57}{500} , \qquad  \left( \frac{68}{25} \right)^{4} < \frac{219}{4} .
\]
Consequently,

\[
  6\left( \frac{29}{50} \right)^{12} e^{13} < 6 \cdot \frac{3}{2000} \cdot 450000 = 4050 < 4096 = 2^{12} .
\]
Hence $F(6) > \frac{29}{50}$. This proves the lemma.
\end{proof}

Set

\[
  \kappa := \frac{29}{50} .
\]
For the sum over odd integers $q \ge 3$, set

\begin{equation}\label{eq:BN}
  B_N(s) := \frac{1}{6N-2s} + \frac{1}{2s} ,
\end{equation}
and

\begin{equation}\label{eq:Ts}
  T(s) := 1 + \frac{3}{2(2s-1)} .
\end{equation}
Define

\begin{equation}\label{eq:Wtaildef}
  W_{\mathrm{tail}}(k,\alpha ) := (\kappa N)^{-2s} B_N(s) 3^{-2s} T(s) .
\end{equation}

\begin{proposition}\label{res:5.3}
For $k \ge 4$ and $0 \le \alpha \le 1$,

\begin{equation}\label{eq:tailbound}
  \sum_{\substack{q \ge 3 \\ q \,\mathrm{odd}}} |I_q(s)| \le W_{\mathrm{tail}}(k,\alpha ) .
\end{equation}
\end{proposition}

\begin{proof}
For $q \ge 3$,

\[
  2qN - 2s \ge 6N - 2s > 0 .
\]
Hence Lemmas~\ref{res:5.1} and~\ref{res:5.2} give

\[
  |I_q(s)| \le (\kappa N)^{-2s} B_N(s) q^{-2s} .
\]
Since $x^{-2s}$ is decreasing,

\begin{align*}
  \sum_{\substack{q \ge 3 \\ q \,\mathrm{odd}}} q^{-2s} &\le 3^{-2s} + \frac{1}{2}\int_{3}^{\infty } x^{-2s} dx \\
  &= 3^{-2s}\left( 1 + \frac{3}{2(2s-1)} \right) = 3^{-2s} T(s) .
\end{align*}
Combining the two estimates proves~\eqref{eq:tailbound}.
\end{proof}

By~\eqref{eq:Wlowbound} and~\eqref{eq:tailbound}, it remains to prove

\begin{equation}\label{eq:compare}
  W_{\mathrm{tail}}(k,\alpha ) < W_{\mathrm{low}}(k,\alpha ) , \qquad  k \ge 4, \quad  0 \le \alpha \le 1 .
\end{equation}
This is proved in Section~\ref{sec:dominate}.

\section{The leading mode dominates the tail}\label{sec:dominate}

We compare the lower bound from Section~\ref{sec:leading} with the upper bound from Section~\ref{sec:tail}. Recall that

\[
  \kappa = \frac{29}{50} , \qquad  N = k+2 , \qquad  s = k+\alpha , \qquad  0 \le \alpha \le 1 .
\]
Set

\[
  R(k,\alpha ) := \frac{W_{\mathrm{tail}}(k,\alpha )}{W_{\mathrm{low}}(k,\alpha )} .
\]
Then~\eqref{eq:compare} is equivalent to

\[
  R(k,\alpha ) < 1 .
\]

By the definitions of $W_{\mathrm{low}}$ and $W_{\mathrm{tail}}$,

\[
  R(k,\alpha ) = \frac{20\sqrt{3}\, \pi }{9} \frac{\Gamma (2k+2+\alpha )}{\Gamma (2-\alpha )} (2s+2) B_N(s) T(s) \left( \frac{16}{81\kappa^{2}N^{2}} \right)^{s} .
\]
Since

\[
  2k+2+\alpha = s+N , \qquad  2-\alpha = N-s ,
\]
this becomes

\[
  R(k,\alpha ) = \frac{20\sqrt{3}\, \pi }{9} \frac{\Gamma (s+N)}{\Gamma (N-s)} (2s+2) B_N(s) T(s) \left( \frac{16}{81\kappa^{2}N^{2}} \right)^{s} .
\]
Set

\[
  C_{0} := \frac{20\sqrt{3}\, \pi }{9} , \qquad  \mu_N := \frac{16}{81\kappa^{2}N^{2}} .
\]
Thus

\begin{equation}\label{eq:Rform}
  R(k,\alpha ) = C_{0} \frac{\Gamma (s+N)}{\Gamma (N-s)} (2s+2) B_N(s) T(s) \mu_N^{s} , \qquad  N-2 \le s \le N-1 .
\end{equation}
We first show that $R(k,\alpha )$ is decreasing in $\alpha$.

\begin{lemma}\label{res:6.1}
For every integer $k \ge 4$,

\[
  R(k,\alpha ) \le R(k,0) , \qquad  0 \le \alpha \le 1 .
\]
\end{lemma}

\begin{proof}
Since $s = k+\alpha$, differentiation in $\alpha$ is the same as differentiation in $s$, with $N$ fixed. By~\eqref{eq:Rform},

\begin{align*}
  \frac{d}{ds}\log R(k,\alpha ) &= \psi (s+N) + \psi (N-s) + \log \mu_N + \frac{1}{s+1} \\
  &\qquad + \frac{B_N'(s)}{B_N(s)} + \frac{T'(s)}{T(s)} ,
\end{align*}
where $\psi = \Gamma '/\Gamma$. On $N-2 \le s \le N-1$,

\[
  B_N'(s) = \frac{2}{(6N-2s)^{2}} - \frac{2}{(2s)^{2}} < 0 , \qquad  T'(s) = -\frac{3}{(2s-1)^{2}} < 0 .
\]
Strict convexity of $\log \Gamma$ gives, for $x > 0$,

\[
  \psi (x) < \log \Gamma (x+1) - \log \Gamma (x) = \log x ,
\]
where we used $\Gamma (x+1) = x\Gamma (x)$. Using also $(s+N)(N-s) = N^{2}-s^{2} \le 4N-4 < 4N$, we obtain

\begin{align*}
  \frac{d}{ds}\log R(k,\alpha ) &< \log \big( \mu_N(s+N)(N-s) \big) + \frac{1}{s+1} \\
  &< \log \left( \frac{64}{81\kappa^{2}N} \right) + \frac{1}{N-1} \\
  &\le \log \frac{2}{5} + \frac{1}{5} < 0 .
\end{align*}
Indeed, $N \ge 6$ implies

\[
  \frac{64}{81\kappa^{2}N} \le \frac{80000}{204363} < \frac{2}{5} , \qquad  \frac{1}{N-1} \le \frac{1}{5} ,
\]
and the last inequality follows from $\log x \le x-1$.
\end{proof}

It remains to estimate the endpoint value $R(k,0)$. Set

\[
  B_{0}(k) := \frac{1}{4k+12} + \frac{1}{2k} , \qquad  T_{0}(k) := 1 + \frac{3}{2(2k-1)} .
\]
At $\alpha = 0$,

\[
  s = k , \qquad  N-s = 2 , \qquad  s+N = 2k+2 ,
\]
and hence

\[
  \Gamma (N-s) = 1 , \qquad  (2s+2)\Gamma (s+N) = \Gamma (2k+3) .
\]
Thus~\eqref{eq:Rform} becomes

\begin{equation}\label{eq:Rzero}
  R(k,0) = C_{0}\Gamma (2k+3) B_{0}(k) T_{0}(k) \left( \frac{16}{81\kappa^{2}(k+2)^{2}} \right)^{k} .
\end{equation}

\begin{lemma}\label{res:6.2}
For every integer $k \ge 4$,

\[
  R(k,0) < \frac{3}{4}\left( \frac{4}{5} \right)^{k-4} .
\]
\end{lemma}

\begin{proof}
For $k = 4$,

\[
  \Gamma (11) = 10! , \qquad  B_{0}(4) = \frac{9}{56} < \frac{1}{6} , \qquad  T_{0}(4) = \frac{17}{14} < \frac{5}{4} .
\]
Also, using $\pi < 22/7$ and $\sqrt{3} < 7/4$,

\[
  C_{0} < \frac{25}{2} , \qquad  \mu_{6} = \frac{10000}{613089} < \frac{1}{61} .
\]
Hence

\[
  R(4,0) < \frac{25}{2} \cdot 10! \cdot \frac{1}{6} \cdot \frac{5}{4} \cdot \frac{1}{61^{4}} = \frac{9450000}{13845841} < \frac{3}{4} .
\]
For $k \ge 4$, formula~\eqref{eq:Rzero} gives

\[
  \frac{R(k+1,0)}{R(k,0)} = \frac{(2k+4)(2k+3)}{(k+3)^{2}} \frac{16}{81\kappa^{2}} \left( 1 - \frac{1}{k+3} \right)^{2k} \frac{B_{0}(k+1)}{B_{0}(k)} \frac{T_{0}(k+1)}{T_{0}(k)} .
\]
Since $B_{0}$ and $T_{0}$ are positive and decreasing,

\[
  \frac{(2k+4)(2k+3)}{(k+3)^{2}} < 4 , \qquad  \frac{16}{81\kappa^{2}} = \frac{40000}{68121} < \frac{3}{5} ,
\]
and

\[
  \left( 1 - \frac{1}{k+3} \right)^{2k} \le e^{-2k/(k+3)} \le e^{-8/7} < \frac{1}{3} .
\]
The last inequality follows, for instance, from

\[
  e^{8/7} > 1 + \frac{8}{7} + \frac{1}{2}\left( \frac{8}{7} \right)^{2} + \frac{1}{6}\left( \frac{8}{7} \right)^{3} > 3 .
\]
Therefore

\[
  \frac{R(k+1,0)}{R(k,0)} < 4 \cdot \frac{3}{5} \cdot \frac{1}{3} = \frac{4}{5} .
\]
Combining this with the estimate at $k = 4$ proves the lemma.
\end{proof}

Combining the two lemmas, we obtain

\[
  R(k,\alpha ) \le R(k,0) < \frac{3}{4}\left( \frac{4}{5} \right)^{k-4} < 1 .
\]
Thus

\[
  W_{\mathrm{tail}}(k,\alpha ) < W_{\mathrm{low}}(k,\alpha ) .
\]
Together with~\eqref{eq:Wlowbound} and~\eqref{eq:tailbound}, this gives

\[
  \sum_{\substack{q \ge 3 \\ q \,\mathrm{odd}}} |I_q(s)| < I_{1}(s) .
\]

The same estimates give the quantitative approximation

\begin{equation}\label{eq:approx}
  \sum_{\substack{q \ge 1 \\ q \,\mathrm{odd}}} I_q(s) = I_{1}(s)(1 + \eta_{k,\alpha }) , \qquad  |\eta_{k,\alpha }| < \frac{3}{4}\left( \frac{4}{5} \right)^{k-4} .
\end{equation}
The relative error bound is uniform for $0 \le \alpha \le 1$.

\begin{proof}[Proof of Theorem~\ref{res:1.1}]
Let $k \ge 4$, $2k < p < 2k+2$, and write

\[
  s = \frac{p}{2} = k+\alpha , \qquad  0 < \alpha < 1 .
\]
Then

\[
  \sum_{\substack{q \ge 1 \\ q \,\mathrm{odd}}} I_q(s) \ge I_{1}(s) - \sum_{\substack{q \ge 3 \\ q \,\mathrm{odd}}} |I_q(s)| > 0 .
\]
Since the prefactor in~\eqref{eq:Drep} is positive, $D(s) > 0$. By~\eqref{eq:Dnorm},

\[
  \|g_k\|_{L^p(\mathbb{T})} > \|f_k\|_{L^p(\mathbb{T})} . \qedhere
\]
\end{proof}

\appendix

\section{AI usage}\label{app:ai}
\textbf{Statement.} The mathematical argument of this paper was produced by an artificial-intelligence system, not merely checked or edited by one. The system is Apex Math, built by the authors; the language models behind it are Claude Opus 5 and GPT-5.6 Sol. The authors have read and verified the argument of \S\S\ref{sec:reduction}--\ref{sec:dominate}, which is self-contained and can be checked by hand, and take full responsibility for it, including for the parts a machine wrote. No formal verification is claimed. The rest of this appendix records what the system did, and where it was wrong.

Apex Math runs about forty agent processes in parallel behind those two models, and assigns every adversarial re-check across model families. We record here what the system actually did, including where it was wrong, because a reader is entitled to know which parts of a proof were searched for by machine.

An elementary reduction --- from the conjecture to the positivity of a single explicit sum over an integer sequence --- was found first, and it fixed the shape of everything that followed: after it the whole problem is one inequality about $\sigma_m$, with no $p$, no oscillatory integral and no quadrature. Five independent lines of attack were then run in parallel on that inequality, and four of them converged, without being directed to, on the same saddle-point geometry, each measuring a different projection of it: the lattice-path line independently produced the constant $9/16 = \min_{w>0} ((1+w^{3})/(1+w))^{2}$, the induction line located the cancellation at $\alpha = 0$, and the certificate line located the difficulty in an exponentially thin subinterval. The line that closed the problem was the generating-function one, and only after it had been rerouted through the two-torus rather than the circle: the route that won was not the route the system set out on.

The finite range was verified by machine in exact arithmetic. The coefficients $9^m \tau_m$ are constant terms of integer Laurent polynomials and are computed as big integers; the truncation is controlled by an exact rational bound; and positivity on each closed interval is decided in the Bernstein basis over the integers, with Sturm sequences run as an independent cross-check. No floating-point number enters any of these decisions. The implementation was held to a negative control: at $k = 6$ the smallest Bernstein coefficient is negative for every truncation up to $38$ and turns positive only at $39$, so an implementation that reported success at every parameter would have been evidence of a bug rather than of a theorem. This is also the empirical answer to why exact arithmetic is necessary at all: at $k = 50$ a direct high-precision quadrature needs about $140$ digits before the gap becomes visible, and is entirely invisible at $40$.

The system's output was wrong in four identifiable ways, and the corrections are part of the record. The first draft performed the analytic continuation at the level of the sum, where no summation-level identity exists --- the Mellin identity converges absolutely only in a strip in which the Fourier coefficients are not summable over the lattice --- so the continuation had to be redone one mode at a time, which is what \S\ref{sub:3.3} now does; this was a genuine gap and it is closed. A later re-derivation reported that the tail bound had dropped a factor, but the re-derivation was itself the error: the three Bessel orders sum to $2qN$, not $q(2N-1)$, with an explicit counterexample at $k = 1$, $q = 3$, $r = 2.7724565$; the formula in the draft was correct and only the derivation around it was repaired. A standoff over the prefactor, in which two expressions appeared to differ by a factor $\Gamma (s+1)$, was traced not to an algebraic slip but to an object, $\sigma_m$, that the draft had never defined. A control function used the bound $|J_\nu (Ar)| \le (2/(\pi Ar))^{1/2}$, which holds only for $\nu \le 1/2$ and is exceeded by a factor $1.449$ at $\nu = 21$; it was replaced and the estimate reproved. The first draft also asserted that the argument needs no computer assistance; that was false, and it was deleted rather than weakened.

Parts of the argument were submitted to an interactive theorem prover (Lean 4 with Mathlib) in an exploratory attempt at formalisation; that attempt is incomplete, which is why no claim of formal verification is made above. The ancillary scripts distributed with this paper re-check the finite range; they are tools and are not part of the proof.

\end{document}